\documentclass[11pt,reqno]{amsart}

\usepackage[T1]{fontenc}
\usepackage{lmodern}
\usepackage{microtype}
\usepackage{mathtools}
\usepackage{enumitem}
\usepackage{xcolor}
\usepackage[colorlinks = true,linkcolor = blue!55!black,citecolor = blue!55!black,urlcolor = blue!55!black]{hyperref}
\usepackage[normalem]{ulem}

\setlist[enumerate]{label = (\roman*),leftmargin = 2.2em}
\setlist[itemize]{leftmargin = 2.0em}
\numberwithin{equation}{section}
\newtheorem{theorem}{Theorem}[section]
\newtheorem{proposition}[theorem]{Proposition}
\newtheorem{lemma}[theorem]{Lemma}
\newtheorem{corollary}[theorem]{Corollary}
\theoremstyle{remark}
\newtheorem{remark}[theorem]{Remark}

\title[A characterization of entrywise positivity preservers]
{A finite-order characterization of entrywise positivity preservers}

\author{Ludovick Bouthat}
\address{Département de mathématiques et de statistique, Université Laval, Québec, QC, Canada}
\email{ludovick.bouthat.1@ulaval.ca}

\author{Dominique Guillot}
\address{Department of Mathematical Sciences, University of Delaware, Newark, DE, USA and Département de mathématiques et de statistique, Université Laval, Québec, QC, Canada}
\email{dguillot@udel.edu}


\begin{document}

\begin{abstract}
Fix $I = (0,\rho)$, where $0<\rho\leq\infty$, and let $\mathbb{P}_n(I)$ be the set of positive semidefinite $n\times n$ matrices with entries in $I$. A longstanding problem in matrix theory is to characterize the functions $f: I \to \mathbb{R}$ for which the entrywise calculus $f[A] = [f(a_{ij})]_{i,j = 1}^{n}$ preserves positive semidefiniteness for all $A \in \mathbb{P}_n(I)$. We characterize these functions exactly: if $f\in C^{2n-2}(I)$ and $\mathcal{E} = x\frac{d}{dx}$, then this holds if and only if
$$
  f^{(k)}(x)\geq 0
  \quad (0\leq k\leq n-1)
  \qquad\text{and}\qquad
  \bigl[\mathcal{E}^{i+j}f(x)\bigr]_{i,j = 0}^{n-1}\succeq 0
$$
for every $x\in I$. Regularization then removes all a priori smoothness: for $n\geq2$, every preserver belongs to $C^{2n-4}(I)$ and the same characterization holds by interpreting the last two derivatives in the sense of distributions.

As applications, we recover classical results of FitzGerald--Horn and Vasudeva, and obtain a complete classification of generalized polynomials with prescribed real exponents and arbitrary coefficients. We also determine optimal constants in entrywise domination inequalities under finite regularity, extend the sharp finite-sum thresholds of Belton--Guillot--Khare--Putinar and Khare--Tao to positive mixtures of powers, and answer a question of Khare and Tao by showing that no finite collection of matrices with entries strictly inside $I$ can detect positivity preservation on $\mathbb{P}_n(I)$.
\end{abstract}

\maketitle

\section{Introduction}

Let $I = (0,\rho)$, where $0<\rho\leq\infty$, and let $\mathbb{P}_n(I)$ denote the set of positive semidefinite $n\times n$ matrices with entries in $I$. Given a function $f:I\to\mathbb{R}$ and a matrix $A = [a_{ij}]_{i,j = 1}^n\in\mathbb{P}_n(I)$, define $f[A] := \bigl[f(a_{ij})\bigr]_{i,j = 1}^n$. A longstanding problem in matrix analysis is to characterize the functions for which
\[
  A\in\mathbb{P}_n(I) \quad\Longrightarrow\quad f[A]\succeq0,
\]
where $\succeq$ denotes the usual positive semidefinite Loewner order. The question goes back at least to Pólya and Szegő in 1925 \cite[Problem~37]{PolyaSzego} and arises naturally in the study of positive definite kernels, metric embeddings, moment problems, harmonic analysis, and high-dimensional statistics; see the surveys \cite{BGKP-SurveyI,BGKP-SurveyII}, the monograph \cite{KhareMatrixAnalysis} and the references therein for more details.

Schur's product theorem \cite{Schur} gives an immediate and fundamental family of such functions. If $f(x) = \sum_{k\geq0}c_kx^k$ is such that $c_k\geq 0$ for all $k\geq 0$, then
\[
  f[A] = \sum_{k\geq0}c_kA^{\circ k}\succeq0,
\]
where $\smash{A^{\circ k} = [a_{ij}^k]_{i,j = 1}^n}$ denotes the $k$th entrywise power. Such functions admitting a power series with nonnegative coefficients are said to be \emph{absolutely monotonic}. When positivity preservation is required in every dimension, this elementary construction is exhaustive. Schoenberg \cite{Schoenberg} proved this under a continuity assumption through his study of positive definite functions on spheres. Rudin \cite{Rudin} later removed continuity using harmonic analysis on the torus, and Vasudeva \cite{Vasudeva} obtained the corresponding result on positive intervals. Thus the dimension-free problem has a complete and rigid answer.

Fixing the dimension changes the problem substantially. Entrywise preservers in dimension $n$ need not be absolutely monotonic, and their Taylor expansions may contain negative coefficients. For example, $1+x-\frac{1}{5}x^2$ preserves positivity on $\mathbb{P}_2((0,1))$. Such examples were discovered in recent work of Belton--Guillot--Khare--Putinar \cite{BeltonGuillotKharePutinar} and Khare--Tao \cite{KhareTao}; see Theorem~\ref{thm:generalized-polynomial-preservers} and Remark~\ref{rem:atomic-christoffel-threshold}. Consequently, the dimension-free classification does not admit a simple truncated version in fixed dimension.

The first general obstruction was discovered by Loewner and developed by Horn~\cite{Horn}: if $f$ is sufficiently smooth and preserves positivity in dimension $n$, then
\[
f^{(k)}(x)\geq0, \qquad 0\leq k\leq n-1.
\]
These inequalities are fundamental, but they do not by themselves provide a characterization. Complete answers were known only in special situations. Vasudeva \cite{Vasudeva} characterized all preservers in dimension two through monotonicity and multiplicative midpoint convexity. FitzGerald and Horn \cite{FitzGeraldHorn} determined exactly which real powers preserve positivity:
\[
  x^\alpha \text{ preserves }\mathbb{P}_n((0,\infty)) \quad\Longleftrightarrow\quad \alpha\in\mathbb{N}_0\cup[n-2,\infty), 
\]
where $\mathbb{N}_0 := \{0\} \cup \mathbb{N}$ denotes the set of nonnegative integers. For arbitrary functions, however, the fixed-dimensional problem has remained open for every $n\geq3$ since Horn's 1969 work with Loewner \cite{Horn}.

Recent work has revealed how much richer the fixed-dimensional cone of preservers is. Belton, Guillot, Khare, and Putinar \cite{BeltonGuillotKharePutinar} obtained sharp coefficient bounds for families of polynomials with a negative term that preserve positivity in fixed dimension. Khare and Tao \cite{KhareTao} determined the possible sign patterns of power series and sums of real powers that preserve positivity, and obtained exact thresholds for a positive sum of powers perturbed by one negative term. Khare's Horn--Loewner theorem \cite{KhareSmooth} further strengthened the known necessary conditions of Horn and Loewner. 

Many other variants were also considered in the literature, involving: structured matrices \cite{BGKP-hankel,GKR-sparse,GuillotKhareRajaratnam}, specific functions \cite{guillot2012retaining,guillot2015complete,GKR-critG, hiai2009monotonicity}, block actions \cite{guillot2015functions,vishwakarma2023positivity}, different notions of positivity \cite{BGKP-tn,belton2026distance, BharaliHoltz, choudhury2025entrywise, guillot2025positivity,guillot2026entrywise, Herz, mashreghi2026functional}, preservation of inertia \cite{belton2023negativity}, and multivariable transforms \cite{belton2023negativity,damase2024multivariate, fitzgerald1995functions}, to name a few. These results provide detailed information on prescribed expansions and special families, but no necessary and sufficient criterion for an arbitrary function to preserve positivity on $\mathbb{P}_n(I)$.

\subsection*{Main results}

The results of this paper provide an explicit characterization of the entrywise  positivity preservers of $\mathbb{P}_n(I)$. We first obtain the characterization in the natural smooth class $C^{2n-2}(I)$, and then remove the regularity assumption entirely via multiplicative regularization and distribution theory. To state our main results, define the \emph{Euler derivative}:
\[
  \mathcal{E} := x\frac{d}{dx}.
\]
Its relevance is transparent following the logarithmic change of variables $g(t) = f(e^t)$: one has $g^{(k)}(t) = \bigl(\mathcal{E}^k f\bigr)(e^t)$. Moreover, rank-one matrices $uu^\top$ become additive Hankel kernels under $g$:
\[
  f(u_i u_j) = g(\log u_i+\log u_j),
\]
which naturally motivates us to consider the \emph{Euler--Hankel matrix}
\[
  \mathcal{H}_n(f;x) := \bigl[\mathcal{E}^{i+j}f(x)\bigr]_{i,j = 0}^{n-1}.
\]

This leads to our main theorem for functions in $C^{2n-2}(I)$. Therein and in what follows, we denote by $\mathbb{P}_n^1(I)$ the set of matrices of rank 1 in $\mathbb{P}_n(I)$.

\begin{theorem}\label{thm:main}
    Fix $n\geq1$, let $0<\rho\leq\infty$, and let $I = (0,\rho)$. Suppose that $f\in C^{2n-2}(I)$. Then the following assertions are equivalent.
    \begin{enumerate}
        \item\label{item:preserver}
        For every $A\in\mathbb{P}_n(I)$, one has $f[A]\succeq0$.
        
        \item\label{item:rank-one-signs}
        For every $R\in\mathbb{P}_n^1(I)$, one has $f[R]\succeq0$, and
        \begin{equation}\label{eq:ordinary-signs}
          f^{(k)}(x)\geq0, \qquad x\in I,\quad 0\leq k\leq n-1.
        \end{equation}
        
        \item\label{item:conditions}
        Condition~\eqref{eq:ordinary-signs} holds and
        \begin{equation}\label{eq:euler-hankel}
          \mathcal{H}_n(f;x) = \bigl[\mathcal{E}^{i+j}f(x)\bigr]_{i,j = 0}^{n-1} \succeq0, \qquad x\in I.
        \end{equation}
    \end{enumerate}
\end{theorem}

\begin{remark}
    The smoothness hypothesis $f\in C^{2n-2}(I)$ is precisely the differentiability needed to state these conditions pointwise. For $n\geq2$, we show in Remark~\ref{rem:distributional-tests} that every preserver actually belongs to $C^{2n-4}(I)$. Hence, the hypothesis $f\in C^{2n-2}(I)$ is not far from being optimal. Regardless, in Theorem~\ref{thm:distributional}, we remove this regularity assumption entirely by interpreting the ordinary and Euler--Hankel derivatives as distributions, thus giving a characterization for arbitrary functions $f$ defined on $I$.
\end{remark}

The equivalence of \ref{item:preserver} and \ref{item:rank-one-signs} is the core of Theorem~\ref{thm:main}: once the Horn--Loewner inequalities \eqref{eq:ordinary-signs} are imposed, positivity on rank-one matrices already determines positivity on all matrices of the same dimension. Proposition~\ref{prop:rank-one}, obtained from Karlin's finite-order criterion after the logarithmic change of variables, identifies rank-one preservation with the Euler--Hankel condition and therefore yields the equivalence of \ref{item:rank-one-signs} and \ref{item:conditions}.

The nontrivial implication from rank-one to full preservation is proved by induction on the dimension. An extension principle of Khare and Tao \cite[Theorem~3.5]{KhareTao} reduces it to preservation by $f'$ in dimension $n-1$, while the new Euler--Hankel descent argument supplies the condition needed to close the induction:
\[
  \mathcal{H}_n(f;x)\succeq 0 \quad\Longrightarrow\quad \mathcal{H}_{n-1}(f';x)\succeq 0.
\]

\smallskip
We note a close formal parallel with the local characterizations of matrix monotonicity and convexity of fixed order obtained in \cite[Theorems~5 and~7]{Heinavaara}.  In that setting, positivity of the Loewner or Kraus matrices for arbitrary tuples of points is equivalent to pointwise positivity of the derivative Hankel matrices
\[
  \left[\frac{f^{(i+j-1)}(t)}{(i+j-1)!}\right]_{i,j = 1}^n \quad\text{or}\qquad \left[\frac{f^{(i+j)}(t)}{(i+j)!}\right]_{i,j = 1}^n,
\]
respectively.  Similarly, after the logarithmic change of variables $g(t) = f(e^t)$, our characterization reduces entrywise positivity preservation, together with the ordinary derivative inequalities, to pointwise positivity of
\[
  \bigl[g^{(i+j)}(t)\bigr]_{i,j = 0}^{n-1}.
\]
Thus, in both problems, a global matrix-preservation property of fixed order is encoded by a finite Hankel matrix of derivatives at each point. The two settings involve different functional calculi and different Hankel matrices, however.

\smallskip
Theorem~\ref{thm:main} naturally yields sharp quantitative consequences. Given sufficiently smooth functions $g$ and $h$, with $h$ satisfying the strict conditions \eqref{eq:ordinary-signs} and \eqref{eq:euler-hankel}, Theorem~\ref{thm:sharp-smooth-domination} determines the least constant $K$ such that
\[
  g[A]\preceq Kh[A] \qquad \text{for every }A\in\mathbb{P}_n(I).
\]
This gives a finite-regularity and optimal form of the domination problem considered by Khare and Tao \cite{KhareTao}.

Specializing the above results to positive mixtures of powers, we recover and extend the sharp finite-sum thresholds of Belton--Guillot--Khare--Putinar and Khare--Tao. Finally, this allows us to give a negative answer to the question in \cite[Remark~8.5]{KhareTao}: no fixed finite collection of test matrices whose entries lie strictly inside $(0,\rho)$ can detect positivity preservation for this interval.

\smallskip

The paper is organized as follows. Section~\ref{sec:preliminaries} collects the finite-order Hankel-kernel criterion and the other preliminary results used below. The rank-one characterization and the Euler--Hankel descent theorem are developed in \ref{sec:rank-one} and \ref{sec:descent}. Section~\ref{sec:proof-main} combines these ingredients to prove the smooth characterization and its distributional extension. Section~\ref{sec:examples} derives several classical fixed- and all-dimensional results. Finally, Section~\ref{sec:sharp-domination} provides applications of Theorem~\ref{thm:main}, namely a sharp domination theorem, a specialization to continuous mixtures of powers, and an obstruction to finite interior testing.

\section{Preliminaries}\label{sec:preliminaries}

We begin by collecting several key results that will be used throughout the subsequent sections.

\subsection{Finite-order additive Hankel kernels}

Let $X\subset\mathbb{R}$ be an open interval and $q$ a function defined on $X+X := \{s+t:s,t\in X\}$.  We say that the additive Hankel kernel $q(s+t)$ is \emph{positive semidefinite of order $m$} if
\[
  [q(s_i+s_j)]_{i,j = 1}^r\succeq 0
\]
for every $1\leq r\leq m$ and every $s_1,\dots,s_r\in X$.  Notice that it is enough to test $r = m$, since smaller matrices occur as principal submatrices after repeating or adjoining points.

For $q\in C^{2m-2}(X+X)$, define its derivative Hankel matrix by
\[
  \mathbf{H}_m(q;t) := \bigl[q^{(i+j)}(t) \bigr]_{\smash{i,j = 0}}^{m-1}, 
\]
where $q^{\smash{(0)}} := q$.

\begin{lemma}[Karlin's criterion]\label{lem:karlin}
    Let $X$ be an open interval and consider a function $q\in C^{2m-2}(X+X)$.
    \begin{enumerate}
        \item If the kernel $q(s+t)$ is positive semidefinite of order $m$, then
        \[
          \mathbf{H}_m(q;t)\succeq 0\qquad(t\in X+X).
        \]
        
        \item If $\mathbf{H}_m(q;t)$ is positive definite for every $t\in X+X$, then the matrix $[q(s_i+s_j)]_{i,j = 1}^m$ is positive definite if the $s_i\in X$ are pairwise distinct.
    \end{enumerate}
\end{lemma}

The necessity (i) follows by replacing point evaluations with divided differences and taking the limit as the nodes get closer to obtain the derivatives. The strict sufficiency in (ii) is a special case of Karlin's differential criterion for kernels; see \cite[Chapter~2, Theorem~2.6]{Karlin}. A convenient recent statement for additive Hankel kernels also appears in \cite[Lemma~4.1]{BodirskyKummerThom}, stated for $X = Y = \mathbb{R}$.

We will later need the following non-strict converse of Karlin's criterion, which follows from a simple regularization.

\begin{corollary}[Semidefinite Karlin criterion]\label{cor:karlin-psd}
    Under the hypotheses of Lemma~\ref{lem:karlin}, if $\mathbf{H}_m(q;t)\succeq 0$ for every $t\in X+X$, then the kernel $q(s+t)$ is positive semidefinite of order $m$.
\end{corollary}

\begin{proof}
Define
\[
  p_m(t) := \sum_{r = 0}^{m-1}e^{rt}, \qquad v_r := (1,r,\dots,r^{m-1})^\top.
\]
A direct computation reveals that
\[
  \mathbf{H}_m(p_m;t) = \sum_{r = 0}^{m-1}e^{rt}v_rv_r^\top.
\]
Since the matrix with columns $v_0,\ldots,v_{m-1}$ is a Vandermonde matrix, these vectors are linearly independent. Therefore, every nonzero $z\in\mathbb R^m$ satisfies
\[
  z^\top \mathbf{H}_m(p_m;t)z = \sum_{r = 0}^{m-1} e^{rt}(v_r^\top z)^2 > 0.
\]
Indeed, equality would imply that $v_r^\top z = 0$ for every $0\leq r\leq m-1$, and hence $z = 0$. Thus $\mathbf{H}_m(p_m;t)$ is positive definite for every $t$. Therefore, $\mathbf{H}_m(q+\varepsilon p_m;t)$ is positive definite for every $\varepsilon>0$.  Lemma~\ref{lem:karlin} gives the desired kernel inequality for $q+\varepsilon p_m$.  Repeated nodes follow by continuity, and then letting $\varepsilon \to 0^+$ gives the assertion for $q$.
\end{proof}

Although it is not used in this paper, we note that the Euler derivatives may be written in terms of ordinary derivatives as
\[
  \mathcal{E}^k f(x) = \sum_{\ell = 0}^k \genfrac{\{}{\}}{0pt}{}{k}{\ell} x^\ell f^{(\ell)}(x),
\]
where the braces denote Stirling numbers of the second kind.  Thus \eqref{eq:euler-hankel} is an explicit finite family of ordinary differential inequalities.

\subsection{Two results on entrywise preservers}

We record two existing results that will play complementary roles in the proof of Theorem~\ref{thm:main}. The first shows that entrywise preservers always have a number of nonnegative derivatives on $I$, while the second provides an inductive mechanism for passing from rank-one matrices to matrices of arbitrary rank.

The first result is the classical smooth Horn--Loewner necessary condition.

\begin{theorem}[Horn--Loewner]\label{thm:horn-loewner}
    Let $n\geq2$ and $f\in C^{n-1}(I)$.  If $f[-]$ preserves positivity on
    $\mathbb{P}_n(I)$, then
    \[
      f^{(k)}(x)\geq0, \qquad x\in I,\quad 0\leq k\leq n-1.
    \]
\end{theorem}

See Horn \cite{Horn} and the formulation and proof in \cite[Theorem~1.1]{KhareSmooth}. These inequalities supply the ordinary derivative conditions in Theorem~\ref{thm:main}. They do not, by themselves, characterize positivity preservation; the additional information will come from the behavior of $f$ on rank-one matrices.

For the converse direction, we use the following extension principle of Khare and Tao \cite[Theorem~3.5]{KhareTao}. 

\begin{proposition}[Extension principle]\label{prop:extension}
    Let $n\geq2$, $I = (0,\rho)$, and $q\in C^1(I)$.  Suppose that
    \begin{enumerate}
        \item $q[R]\succeq 0$ for all $R\in\mathbb{P}_n^1(I)$;
        \item $q'[C]\succeq 0$ for all $C\in\mathbb{P}_{n-1}(I)$.
    \end{enumerate}
    Then $q[A]\succeq 0$ for all $A\in\mathbb{P}_n(I)$.
\end{proposition}

Together, Theorem~\ref{thm:horn-loewner} and Proposition~\ref{prop:extension} reduce the main problem to two questions: how to characterize positivity preservation on rank-one matrices, and how to pass the resulting condition from $f$ to $f'$. The first question is answered by the additive Hankel-kernel criterion, while the second is the content of the Euler--Hankel descent theorem proved below.

\section{Rank-one preservation and Euler--Hankel descent}\label{sec:euler-hankel}

\subsection{The rank-one problem}\label{sec:rank-one}

Let $0<\rho\leq\infty$, and $I = (0,\rho)$. Define
\[
  J := (-\infty,\log\rho),\qquad 
  X := (-\infty,\tfrac12\log\rho),\qquad
  g(t) := f(e^t),
\]
with the convention $\log\infty := \infty$.  Then $X+X = J$ and
\begin{equation}\label{eq:euler-log-identity}
  g^{(k)}(t) = \mathcal{E}^kf(e^t),\qquad k\geq0.
\end{equation}
Moreover, for clarity, we write
\[
  \mathcal{H}_n(f;e^t) := \mathbf{H}_n(g;t).
\]
The following connects the Euler--Hankel condition and positivity preservers on rank-one matrices.

\begin{proposition} \label{prop:rank-one}
    Let $n\geq 1$ and $f\in C^{2n-2}(I)$.  The following are equivalent.
    \begin{enumerate}
      \item $f[R]\succeq 0$ for all $R\in\mathbb{P}_n^1(I)$.
      \item $\mathcal{H}_n(f;x)\succeq 0$ for all $x\in I$.
    \end{enumerate}
\end{proposition}

\begin{proof}
Suppose first that (ii) holds.  By \eqref{eq:euler-log-identity}, $\mathbf{H}_n(g;t)\succeq 0$ on $J$. Hence, Corollary~\ref{cor:karlin-psd} shows that $g(s+t)$ is positive semidefinite of order $n$ on $X$.

Now, let $R\in\mathbb{P}_n^1(I)$ and write $R = uu^\top$.  Since the entries of $R$ are positive, the coordinates of $u$ have the same sign. Replacing $u$ by $-u$ if necessary, assume $u_i>0$.  Moreover, $u_i^2<\rho$, so $s_i := \log u_i$ belongs to $X$.  Therefore
\[
  f[R] = [f(u_iu_j)]_{i,j = 1}^n = [g(s_i+s_j)]_{i,j = 1}^n \succeq 0.
\]
Conversely, rank-one preservation gives the last matrix inequality for every $s_1,\dots,s_n\in X$.  The necessary half of Lemma~\ref{lem:karlin}, followed by \eqref{eq:euler-log-identity}, yields $\mathcal{H}_n(f;x)\succeq 0$ for every $x\in I$.
\end{proof}

\subsection{Descent under differentiation}\label{sec:descent}

The next theorem provides the main technical ingredient for the proof of our main result. As we see, the fact that $J = (-\infty,\log\rho)$ is unbounded to the left is essential in its proof.

\begin{theorem}[Euler--Hankel descent]\label{thm:descent}
    Let $m\geq2$, let $I = (0,\rho)$, and suppose $f\in C^{2m}(I)$ satisfies $f,f',f''\geq0$ on $I$.  If
    \[
      \mathcal{H}_{m+1}(f;x)\succeq 0\qquad(x\in I),
    \]
    then
    \begin{equation}\label{eq:descent-output}
      \mathcal{H}_m(f';x)\succeq 0\qquad(x\in I).
    \end{equation}
\end{theorem}

\begin{proof}
Let $g(t) = f(e^t)$ on $J = (-\infty,\log\rho)$.  The principal submatrix of $\mathcal{H}_{m+1}(f;e^t)$ indexed by $1,\dots,m$ (in the lower-right corner) is
\[
  \bigl[g^{(i+j+2)}(t)\bigr]_{i,j = 0}^{m-1} = \mathbf{H}_m(g'';t) \succeq 0.
\]
By Corollary~\ref{cor:karlin-psd}, the kernel $g''(s+t)$ is positive semidefinite of order $m$ on $X = (-\infty,\frac12\log\rho)$.

Now, direct differentiation of $g(t) = f(e^t)$ gives
\[
  g'(t) = e^t f'(e^t)\geq0, \qquad g''(t)-g'(t) = e^{2t}f''(e^t)\geq0.
\]
In particular, $g'' \geq g' \geq 0$, so $g'$ is nondecreasing. Therefore the finite, nonnegative limit $\ell := \lim_{t\to-\infty}g'(t)$ exists.  We claim $\ell = 0$.  Indeed, if $\ell>0$, then for fixed $T\in J$ and $t<T$,
\[
  g(t) = g(T)-\int_t^Tg'(u) \,\mathrm{d}u \leq g(T)-\ell(T-t),
\]
which is negative for large enough negative $t$, contradicting $g(t) = f(e^t)\geq0$.  Hence
\begin{equation}\label{eq:g-prime-integral}
  g'(x) = \int_{-\infty}^x g''(u)\,\mathrm{d}u = \int_0^\infty g''(x-t)\,\mathrm{d}t.
\end{equation}

Fix $s_1,\dots,s_m\in X$.  For $T>0$, the matrix
\[
  \int_0^T \bigl[g''((s_i-t/2)+(s_j-t/2))\bigr]_{i,j = 1}^m\,\mathrm{d}t
\]
is positive semidefinite since every integrand is positive semidefinite, as $s_i-t/2\in X$.  Letting $T\to\infty$ and using \eqref{eq:g-prime-integral} entrywise yields
\begin{equation}\label{eq:g-prime-kernel}
  [g'(s_i+s_j)]_{i,j = 1}^m\succeq 0.
\end{equation}
Thus $g'(s+t)$ is positive semidefinite of order $m$. Set $h(t) := e^{-t}g'(t) = f'(e^t).$ For $D = \operatorname{diag}(e^{-s_1},\dots,e^{-s_m})$, Equation
\eqref{eq:g-prime-kernel} gives the diagonal congruence
\[
  [h(s_i+s_j)]_{i,j = 1}^m = D[g'(s_i+s_j)]_{i,j = 1}^mD\succeq 0.
\]
The necessary half of Lemma~\ref{lem:karlin} now implies $\mathbf{H}_m(h;t)\succeq 0$ on $J$.  Finally, $h^{(k)}(t) = \mathcal{E}^kf'(e^t)$, so this is exactly \eqref{eq:descent-output}.
\end{proof}

\begin{remark}\label{rem:descent-kernel-mechanism}
The preceding proof uses two operations preserving finite-order kernel positivity. First, $g'(x) = \int_0^\infty g''(x-t)\,\mathrm{d}t $ implies, for $s_1,\ldots,s_m\in X$,
\[
  \bigl[g'(s_i+s_j)\bigr]_{i,j = 1}^m = \int_0^\infty \bigl[g''\bigl((s_i-t/2)+(s_j-t/2)\bigr) \bigr]_{i,j = 1}^m \,\mathrm{d}t.
\]
Thus the kernel associated with $g'$ is a positive mixture of translates of the kernel associated with $g''$. Since $X = (-\infty,\frac12\log\rho)$ is invariant under negative translations, every matrix in the integrand is positive semidefinite.

Second, since $f'(e^{s_i+s_j}) = e^{-s_i}e^{-s_j}g'(s_i+s_j), $ we have
\[
  \bigl[f'(e^{s_i+s_j})\bigr]_{i,j = 1}^m = D\bigl[g'(s_i+s_j)\bigr]_{i,j = 1}^m D, \qquad D = \operatorname{diag}(e^{-s_1},\ldots,e^{-s_m}).
\]
The passage from $g'$ to $f'\circ\exp$ is therefore a positive diagonal congruence.

Consequently, the result is the composition of two kernel-positivity operations: integration against positive translates and multiplication by a positive separable factor. The conditions $f',f''\geq0$ ensure that $g'$ is nonnegative and nondecreasing, while $f\geq0$ and the left-unboundedness of $J$ force
\[
  \lim_{t\to-\infty}g'(t) = 0,
\]
which is precisely the condition needed for the integral representation above.
\end{remark}

\section{Proof of the main theorem}\label{sec:proof-main}

\subsection{The smooth case}

We now return to the fixed-dimensional preserver problem and prove Theorem~\ref{thm:main}. Its smoothness assumption will be removed in the following subsection.

\begin{proof}[Proof of Theorem~\ref{thm:main}]
Proposition~\ref{prop:rank-one} shows immediately that \ref{item:rank-one-signs} and \ref{item:conditions} are equivalent. It therefore remains to prove the equivalence of \ref{item:preserver} and \ref{item:rank-one-signs}.

If $n = 1$, all three assertions reduce to $f\geq0$ on $I$. Assume henceforth that $n\geq2$. Suppose first that \ref{item:preserver} holds. Positivity is then preserved, in particular, on the rank-one subset $\mathbb{P}_n^1(I)$. Moreover, the Horn--Loewner theorem gives \eqref{eq:ordinary-signs}. Hence \ref{item:rank-one-signs} holds.

Conversely, suppose that \ref{item:rank-one-signs} holds. We prove \ref{item:preserver} by induction on $n$. Let $n = 2$. By hypothesis, $f$ preserves positivity on $\mathbb{P}_2^1(I)$. Since $f'\geq0$, the function $f'$ preserves positivity on $\mathbb{P}_1(I)$. Proposition~\ref{prop:extension} therefore shows that $f$ preserves positivity on $\mathbb{P}_2(I)$.

Now let $n\geq3$ and assume the result in dimension $n-1$. Proposition~\ref{prop:rank-one} gives
\[
  \mathcal{H}_n(f;x)\succeq 0, \qquad x\in I.
\]
Applying Theorem~\ref{thm:descent} with $m = n-1$ yields
\[
  \mathcal{H}_{n-1}(f';x)\succeq 0, \qquad x\in I.
\]
Proposition~\ref{prop:rank-one}, now applied to $f'$ in dimension $n-1$, shows that $f'$ preserves positivity on $\mathbb{P}_{n-1}^1(I)$. Moreover,
\[
  (f')^{(k)} = f^{(k+1)}\geq0, \qquad 0\leq k\leq n-2,
\]
and
\[
  f'\in C^{2n-3}(I)\subset C^{2n-4}(I).
\]
Thus $f'$ satisfies assertion \ref{item:rank-one-signs} in dimension $n-1$. By the induction hypothesis, $f'$ preserves positivity on $\mathbb{P}_{n-1}(I)$. Finally, Proposition~\ref{prop:extension} shows that $f$ preserves positivity on $\mathbb{P}_n(I)$.
\end{proof}

\subsection{Removal of the smoothness hypothesis}\label{sec:nonsmooth}

The smoothness assumption in Theorem~\ref{thm:main} can be removed by regularizing multiplicatively and interpreting its derivatives in the sense of distributions.  If $U\subset\mathbb{R}$ is open, let
\[
  \mathcal{D}(U) := C_c^\infty(U)
\]
denote the space of smooth, compactly supported functions on $U$, and let
\[
  \mathcal{D}'(U) := \mathcal{D}(U)'
\]
be its continuous dual, i.e., the \emph{space of distributions} on $U$.  We denote the action of a distribution $T\in\mathcal{D}'(U)$ on a function $\varphi\in\mathcal{D}(U)$ by $\langle T,\varphi\rangle$.  Let
\[
  L_{\mathrm{loc}}^1(U) := \left\{h:U\to\mathbb{R} \text{ measurable}:\, \forall K\subset U \text{ compact},\, \int_K |h(x)|\,\mathrm{d}x<\infty \right\}
\]
be the space of \emph{locally integrable functions} on $U$. Every function $h\in L_{\mathrm{loc}}^1(U)$ is identified with the distribution defined by
\[
  \langle h,\varphi\rangle := \int_U h(x)\varphi(x)\,\mathrm{d}x.
\]
Its $k$th distributional derivative $D^kh\in\mathcal{D}'(U)$ is then defined by
\[
  \langle D^kh,\varphi\rangle := (-1)^k\int_U h(x)\varphi^{(k)}(x)\,\mathrm{d}x.
\]
A distribution $T\in\mathcal{D}'(U)$ is called \emph{positive}, and we write $T\geq0$, if
\[
  \langle T,\varphi\rangle\geq0
\]
for every nonnegative $\varphi\in C_c^\infty(U)$.  Every positive distribution is represented by a positive Radon measure (see, e.g., \cite[Theorem~2.1.7]{Hormander}).

For a locally integrable function $h$ on $U$, let $D^kh$ be its $k$th distributional derivative.  Given $g\in L^1_{\mathrm{loc}}(J)$, define the distributional Hankel matrix by
\begin{equation}\label{eq:distributional-hankel}
  \mathbf{H}_n^{\mathrm{dist}}(g) := \bigl[ D^{i+j}g \bigr]_{i,j = 0}^{n-1}.
\end{equation}
We write $\mathbf{H}_n^{\mathrm{dist}}(g)\succeq0$ when
\begin{equation}\label{eq:distributional-hankel-positive}
  \bigl[\langle D^{i+j}g,\psi\rangle\bigr]_{i,j = 0}^{n-1}\succeq 0
\end{equation}
for every nonnegative $\psi\in C_c^\infty(J)$.

The following result extends Theorem~\ref{thm:main} to arbitrary functions.
\begin{theorem}\label{thm:distributional}
    Fix $n\geq2$ and $0<\rho\leq\infty$, let $I = (0,\rho)$, $J = (-\infty,\log\rho)$, and let $f:I\to\mathbb{R}$. Define $g:J\to\mathbb{R}$ by $g(t) = f(e^t)$.  Then the following assertions are equivalent.
    \begin{enumerate}
        \item\label{item:distributional-preserver}
        For every $A\in\mathbb{P}_n(I)$, one has $f[A]\succeq 0$.
        \item\label{item:distributional-conditions}
        The function $f$ is continuous,
        \begin{equation}\label{eq:distributional-ordinary-signs}
          D^kf\geq0\quad\text{in}\quad\mathcal{D}'(I), \qquad 0\leq k\leq n-1,
        \end{equation}
        and
        \begin{equation}\label{eq:distributional-euler-hankel}
          \mathbf{H}_n^{\mathrm{dist}}(g)\succeq0.
        \end{equation}
    \end{enumerate}
\end{theorem}

\begin{proof}
Assume first that \ref{item:distributional-preserver} holds.  Repeating rows and columns shows that $f[-]$ also preserves $\mathbb{P}_2(I)$.  Vasudeva's theorem \cite[Theorem 2]{Vasudeva} therefore implies that $f$ is continuous on $I$.

Choose $\eta\in C_c^\infty((0,1))$ such that $\eta\geq0$ and $\int_0^1\eta(s)\,\mathrm{d}s = 1$, and set
\[
  \eta_\varepsilon(s) := \varepsilon^{-1}\eta(s/\varepsilon).
\]
Define the one-sided multiplicative regularization
\[
  f_\varepsilon(x) := \int_0^\varepsilon \eta_\varepsilon(s)f(e^{-s}x)\,\mathrm{d}s, \qquad x\in I.
\]
By the standard convolution regularization theorem, the convolution of a locally integrable function with a compactly supported smooth function is smooth, and differentiation may be passed to the smooth factor (see e.g.~\cite[Section 5.2]{FriedlanderJoshi}). Hence, $f_\varepsilon$ is a smooth function. Equivalently, if $g_\varepsilon(t) := f_\varepsilon(e^t)$, then
\[
  g_\varepsilon(t) = \int_0^\varepsilon\eta_\varepsilon(s)g(t-s)\,\mathrm{d}s.
\]
The interval $J$ is stable under negative translations, so $g_\varepsilon$, and hence $f_\varepsilon$, is smooth on its entire domain.

If $A\in\mathbb{P}_n(I)$ and $s>0$, then $e^{-s}A\in\mathbb{P}_n(I)$.  Consequently,
\[
  f_\varepsilon[A] = \int_0^\varepsilon \eta_\varepsilon(s)f[e^{-s}A] \,\mathrm{d}s \succeq 0.
\]
Theorem~\ref{thm:main}, applied to $f_\varepsilon$, gives
\begin{equation}\label{eq:regularized-smooth-conditions}
  f_\varepsilon^{(k)}\geq0\quad(0\leq k\leq n-1), \qquad \bigl[g_\varepsilon^{(i+j)}\bigr]_{i,j = 0}^{n-1}\succeq 0.
\end{equation}
As $\varepsilon\to 0^+$, the functions $f_\varepsilon$ converge to $f$ locally uniformly on $I$, and $g_\varepsilon$ converge to $g$ locally uniformly on $J$.  Distributional differentiation is continuous, and the cone of positive distributions is closed.  Passing to the limit in \eqref{eq:regularized-smooth-conditions} therefore proves both \eqref{eq:distributional-ordinary-signs} and \eqref{eq:distributional-euler-hankel}.

Conversely, suppose that \ref{item:distributional-conditions} holds and form $f_\varepsilon$ as above. For $s>0$, let $S_sh(x) := h(e^{-s}x)$.  The distributional chain rule gives
\[
  D^k(S_sf) = e^{-ks}S_s(D^kf).
\]
Pullback by a positive dilation preserves positive distributions. Therefore, it follows from \eqref{eq:distributional-ordinary-signs} and the identity above that
\[
  D^kf_\varepsilon = \int_0^\varepsilon \eta_\varepsilon(s)e^{-ks}S_s(D^kf)\,\mathrm{d}s \geq0.
\]
Since $f_\varepsilon$ is smooth, this says pointwise that
\begin{equation}\label{eq:regularized-ordinary-signs}
  f_\varepsilon^{(k)}(x)\geq0, \qquad 0\leq k\leq n-1.
\end{equation}
Now, fix $t\in J$ and write $\varphi_{t,\varepsilon}(u) := \eta_\varepsilon(t-u).$ Since $\eta_\varepsilon\geq0$ and $\operatorname{supp}\eta_\varepsilon\subset(0,\varepsilon)$, we have $\varphi_{t,\varepsilon}\in C_c^\infty(J)$ and $\varphi_{t,\varepsilon}\geq0.$ Moreover, for every integer $k\geq0$, it follows from the definitions that
\[
  g_\varepsilon^{(k)}(t) = \bigl\langle D^kg,\,\varphi_{t,\varepsilon}\bigr\rangle.
\]
Let $\xi = (\xi_0,\ldots,\xi_{n-1})^\top\in\mathbb{R}^n$.  Using the above identity, we obtain
\begin{align*}
  \xi^\top \big[g_\varepsilon^{(i+j)}(t)\bigr]_{i,j = 0}^{n-1}\xi = \xi^\top \bigl[\langle D^{i+j}g,\varphi_{t,\varepsilon}\rangle\bigr]_{i,j = 0}^{n-1}\xi.
\end{align*}
Since $\varphi_{t,\varepsilon}$ is a nonnegative test function, \eqref{eq:distributional-hankel-positive} and \eqref{eq:distributional-euler-hankel} ensure that
\[
  \xi^\top \bigl[g_\varepsilon^{(i+j)}(t)\bigr]_{i,j = 0}^{n-1}\xi \geq0 \qquad\text{for every }\xi\in\mathbb{R}^n.
\]
Therefore
\begin{equation}\label{eq:regularized-euler-hankel}
  \bigl[g_\varepsilon^{(i+j)}(t)\bigr]_{i,j = 0}^{n-1}\succeq0.
\end{equation}
Since  $ g_\varepsilon^{(i+j)}(t) = \mathcal{E}^{i+j} f_\varepsilon(e^t),$ Equations \eqref{eq:regularized-ordinary-signs} and \eqref{eq:regularized-euler-hankel}, together with Theorem~\ref{thm:main}, show that $f_\varepsilon[-]$ preserves $\mathbb{P}_n(I)$.  Finally, $f_\varepsilon\to f$ locally uniformly, and the cone of positive semidefinite matrices is closed.  Hence, for every $A\in\mathbb{P}_n(I)$, the convergence $f_\varepsilon[A]\to f[A]$ implies $f[A]\succeq 0$, proving \ref{item:distributional-preserver}.
\end{proof}

\begin{remark}\label{rem:distributional-tests}
The conditions in Theorem~\ref{thm:distributional} involve no pointwise derivatives.  Explicitly, \eqref{eq:distributional-ordinary-signs} is equivalent to
\[
  (-1)^k\int_I f(x)\varphi^{(k)}(x)\,\mathrm{d}x \geq 0 \qquad (0\leq k \leq n-1)
\]
for every nonnegative $\varphi\in C_c^\infty(I)$, while \eqref{eq:distributional-euler-hankel} is equivalent to
\[
  \left[ (-1)^{i+j}\int_J g(t)\psi^{(i+j)}(t)\,\mathrm{d}t \right]_{i,j = 0}^{n-1}\succeq 0
\]
for every nonnegative $\psi\in C_c^\infty(J)$. Moreover, the bottom-right entry of \eqref{eq:distributional-hankel} shows that $D^{2n-2}g$ is a positive Radon measure. It follows that $D^{2n-3} g$ has a locally BV representative and $D^{2n-4}g$ has a locally absolutely continuous representative (see, e.g., \cite[Theorem~3.29 and Corollary~3.33]{Folland}). Thus all derivatives up to order $2n-4$ are classical and continuous.
\end{remark}

\section{Recovering classical theorems}\label{sec:examples}

We now show how Theorems \ref{thm:main} and \ref{thm:distributional} can be used to recover and extend many classical and recent characterizations of entrywise preservers.

\subsection{Critical exponent of \texorpdfstring{$x^{\alpha}$}{xᵅ}}

Let $f(x) = x^\alpha$ on $I = (0,\infty)$.  Then
\[
  \mathcal{H}_n(f;x) = x^\alpha v_\alpha v_\alpha^\top, \qquad v_\alpha = (1,\alpha,\dots,\alpha^{n-1})^\top,
\]
so the Euler--Hankel condition always holds.  For $n\geq2$, the ordinary derivative conditions hold exactly when 
\[
  \alpha\in\mathbb{N}_0\cup[n-2,\infty).
\]
Theorem~\ref{thm:main} therefore recovers the FitzGerald--Horn theorem \cite[Theorem~2.2]{FitzGeraldHorn} for matrices with positive entries.

\subsection{Vasudeva's theorem for \texorpdfstring{$n = 2$}{n = 2}}

For $f\in C^2(I)$,
\[
  \mathcal{H}_2(f;x) = 
  \begin{bmatrix}
    f(x) & xf'(x)\\
    xf'(x) & xf'(x)+x^2f''(x)
  \end{bmatrix}.
\]
Theorem~\ref{thm:main} therefore yields the following. 

\begin{corollary}\label{cor:two-by-two}
Let $f\in C^2(I)$.  Then $f[-]$ preserves positivity on $\mathbb{P}_2(I)$ if and only if $f,f'\geq 0$ and
\begin{equation}\label{eq:two-by-two-determinant}
  f(x)\bigl(xf'(x)+x^2f''(x)\bigr)-x^2f'(x)^2\geq0
\end{equation}
for every $x\in I$.
\end{corollary}

On any subinterval on which $f>0$, Equation \eqref{eq:two-by-two-determinant} is equivalent to $(\log(f(e^t)))''\geq0,$ while $f'\geq0$ says that $f(e^t)$ is nondecreasing.  Thus the criterion is precisely the differential form of Vasudeva's characterization \cite[Theorem 2]{Vasudeva}: nonnegativity, monotonicity, and multiplicative midconvexity
\[
  f(\sqrt{xy})^2\leq f(x)f(y).
\]

% \subsection{Entrywise preservers in all dimensions}

% Theorem~\ref{thm:main}, together with its distributional extension, recovers the Schoenberg--Rudin--Vasudeva classification of entrywise preservers in all dimensions on positive intervals.

% \begin{corollary}\label{cor:schoenberg-rudin-necessity}
%     Let $I = (0,\rho)$, where $0<\rho\leq\infty$, and let $f:I\to\mathbb{R}$.  Then
%     \[
%       f[A]\succeq0 \qquad \text{for every }A\in\mathbb{P}_n(I) \text{ and every }n\geq1
%     \]
%     if and only if $f$ is absolutely monotone on $I$.
% \end{corollary}

% \begin{proof}
% Sufficiency follows from the Schur product theorem. We prove necessity. 

% Fix $n\geq2$.  Since $f[-]$ preserves $\mathbb{P}_n(I)$, Theorem~\ref{thm:distributional} and Remark~\ref{rem:distributional-tests} imply that $f\in C^{2n-4}(I)$. As $n$ is arbitrary, we obtain $f\in C^\infty(I)$. Fix an integer $k\geq0$ and choose $n\geq k+1$.  Applying Theorem~\ref{thm:main} in dimension $n$ gives
% \[
%   f^{(k)}(x)\geq0 \qquad (x\in I).
% \]
% Since $k$ was arbitrary, $f$ is absolutely monotone on $I$. Bernstein's theorem on absolutely monotone functions \cite{Bernstein} therefore yields a power-series representation $ f(x) = \sum_{k = 0}^\infty c_kx^k $ on $I$, with $c_k\geq0$ for every $k\geq0$.
% \end{proof}

\subsection{Polynomials and sums of real powers}

We next use Theorem~\ref{thm:main} to obtain a complete classification for generalized polynomials with prescribed real exponents, extending the sharp quantitative results of Khare and Tao \cite[Theorems~1.11 and~1.12]{KhareTao} by allowing arbitrary real coefficients, including vanishing coefficients. We adopt the same notation as in \cite{KhareTao}: for any vector $\mathbf{u} \in \mathbb{R}^N$, we let $V(\mathbf{u})$ denote the Vandermonde determinant: 
\[
  V(\mathbf{u}) := \det \left(\mathbf{u}^{\circ 0} | \dots | \mathbf{u}^{\circ (N-1)}\right), 
\]
where $\mathbf{u}^{\circ i} := (u_1^i, \dots, u_N^i)^T$ is the $i$-th entrywise power of $\mathbf{u} = (u_1, \dots, u_N)^T$. For any positive integer $r$ and $\alpha \in \mathbb{R}$, we also write $(\alpha)_r := \alpha(\alpha-1)\cdots(\alpha-r+1)$ and let $(\alpha)_0 := 1$.

\begin{theorem}\label{thm:generalized-polynomial-preservers}
    Fix an integer $N\geq2$ and real powers $n_0<\cdots<n_{N-1}<M$. Also fix $\rho>0$ and $c_{n_0},\ldots,c_{n_{N-1}},c'\in\mathbb{R}$, and define $f(x) := \sum_{j = 0}^{N-1}c_{n_j}x^{n_j}+c'x^M.$ Let $\mathcal{A}_N := \mathbb{N}_0\cup[N-2,\infty)$ and, if $c_{n_j}>0$ for every $j$, define
    \begin{equation}\label{eq:generalized-polynomial-critical-constant}
      \mathcal{C} := \sum_{j = 0}^{N-1} \frac{V(\mathbf{n}_j)^2}{V(\mathbf{n})^2} \frac{\rho^{M-n_j}}{c_{n_j}},
    \end{equation}
    where $\mathbf{n} := (n_0,\ldots,n_{N-1})$ and $ \mathbf{n}_j := (n_0,\ldots,n_{j-1},n_{j+1},\ldots,n_{N-1},M)$. Then the following are equivalent:
    \begin{enumerate}
        \item For every $A \in \mathbb{P}_N((0,\rho))$, one has $f[A] \succeq 0$;
        \item one of the following two mutually exclusive conditions holds:
        \begin{enumerate}
          \item[(1)]\label{item:all-nonnegative-generalized} All the coefficients are nonnegative and
          \begin{equation}\label{eq:active-exponents-admissible}
            c_{n_j}>0\Longrightarrow n_j\in\mathcal{A}_N, \qquad c'>0\Longrightarrow M\in\mathcal{A}_N.
          \end{equation}
          \item[(2)]\label{item:one-negative-generalized} For all $0\leq j \leq N-1$, we have $c_{n_j}>0$, $n_j\in\mathcal{A}_N$, and
          \begin{equation}\label{eq:negative-generalized-threshold}
            -\mathcal{C}^{-1}\leq c'<0.
          \end{equation}
        \end{enumerate}
    \end{enumerate}
    Moreover, in case (2), $f[-]$ does not preserve $\mathbb{P}_{N+1}((0,\rho))$.
\end{theorem}

\begin{proof}
Since every real power is smooth on $(0,\rho)$, Theorem~\ref{thm:main} shows that $f[-]$ preserves $\mathbb{P}_N((0,\rho))$ if and only if
\[
  f^{(r)}(x)\geq0 \quad (0\leq r\leq N-1) \qquad\text{and}\qquad \mathcal{H}_N(f;x)\succeq0
\]
for every $x\in(0,\rho)$.

We first characterize the Euler--Hankel condition. Set
\[
  v_\alpha = (1,\alpha,\ldots,\alpha^{N-1})^\top, \qquad W = [\,v_{n_0}\mid\cdots\mid v_{n_{N-1}}\,].
\]
Then $W$ is invertible, and Lagrange interpolation gives
\[
  v_M = Wa,\qquad a_j = \prod_{k\ne j}\frac{M-n_k}{n_j-n_k}, \qquad a_j^2 = \frac{V(\mathbf{n}_j)^2}{V(\mathbf{n})^2}.
\]
Consequently,
\[
  \mathcal{H}_N(f;x) = W\bigl(D_x+c'x^Maa^\top\bigr)W^\top\!,\, \quad D_x = \operatorname{diag}(c_{n_0}x^{n_0},\ldots,c_{n_{N-1}}x^{n_{N-1}}).
\]

If $\mathcal{H}_N(f;x)\succeq0$ for every $x$, its diagonal entries, after division by $x^{n_j}$ and passage to $x\to0^+$, show that $c_{n_j}\geq0$ for every $j$.  Hence, when $c'\geq0$, the Euler--Hankel condition is equivalent simply to
\[
  c_{n_0},\ldots,c_{n_{N-1}},c'\geq0.
\]
If $c'<0$, none of the $c_{n_j}$ can vanish, so $D_x\succ0$; the rank-one perturbation criterion gives
\[
  \mathcal{H}_N(f;x)\succeq0 \Longleftrightarrow -c'\sum_{j = 0}^{N-1} \frac{a_j^2}{c_{n_j}}x^{M-n_j} \leq 1.
\]
Since $M>n_j$ for every $j$, the sum is increasing in $x$.  Thus the Euler--Hankel condition holds for all $x\in(0,\rho)$ if and only if
\[
  c_{n_j}>0\quad(0\leq j\leq N-1), \qquad -\mathcal{C}^{-1}\leq c'<0.
\]

It remains to impose the ordinary derivative inequalities.  We claim that, conditional on the preceding coefficient alternatives, they are equivalent to requiring every exponent carrying a positive coefficient to belong to
\[
  \mathcal{A}_N = \mathbb N_0\cup[N-2,\infty),
\]
except that no condition on $M$ need initially be imposed when $c'<0$.

Indeed, if $\alpha<0$ is the least exponent carrying a positive coefficient, then
\[
  f'(x) = c_\alpha\alpha x^{\alpha-1}+o(x^{\alpha-1})<0
\]
near zero.  Likewise, if $\alpha\in(0,N-2)\setminus\mathbb N$ is the least nonintegral exponent carrying a positive coefficient, then, with $m = \lfloor\alpha\rfloor$ and $r = m+2\leq N-1$,
\[
  f^{(r)}(x) = c_\alpha(\alpha)_r x^{\alpha-r} + o(x^{\alpha-r}) < 0
\]
near zero.  This proves the necessity of the stated exponent conditions.

Conversely, in the nonnegative-coefficient case the conclusion is immediate, since
\[
  (\alpha)_r\geq0 \qquad (\alpha\in\mathcal{A}_N,\ 0\leq r\leq N-1).
\]
Consider therefore the case $c'<0$.  Since $n_0,\ldots,n_{N-1}$ are $N$ distinct elements of $\mathcal{A}_N$, necessarily $n_{N-1}\geq N-2$, and hence $M>N-2$.

Fix $0\leq r\leq N-1$, and write
\[
  b_j = (n_j)_r,\qquad B = (M)_r.
\]
Then $0\leq b_j\leq B$, and, since $t\mapsto(t)_r$ has degree at most $N-1$, the same Lagrange interpolation coefficients as above give
\[
  B = \sum_{j = 0}^{N-1}a_jb_j.
\]
Set $w_j = c_{n_j}\rho^{n_j-M}$. By Cauchy--Schwarz,
\[
  B^2 \leq \left(\sum_{j = 0}^{N-1}\frac{a_j^2}{w_j}\right)\!\left(\sum_{j = 0}^{N-1}w_jb_j^2\right) \leq \mathcal{C} B\sum_{j = 0}^{N-1}w_jb_j.
\]
Thus
\[
  \mathcal{C}^{-1} \leq \frac{1}{(M)_r} \sum_{j = 0}^{N-1} c_{n_j}(n_j)_r\rho^{n_j-M}.
\]
Since $-c'\leq\mathcal{C}^{-1}$ and $x^{n_j-M}\geq\rho^{n_j-M}$ for $0<x<\rho$, it follows that
\[
  -c'(M)_r \leq \sum_{j = 0}^{N-1} c_{n_j}(n_j)_r x^{n_j-M},
\]
and hence $f^{(r)}(x)\geq0$.  This proves the characterization.

Finally, in case (2), $f[-]$ does not preserve positivity in dimension $N+1$. Indeed, with
\[
  \widetilde v_\alpha = (1,\alpha,\ldots,\alpha^N)^\top,
\]
choose $0\ne z\perp\widetilde v_{n_0},\ldots,\widetilde v_{n_{N-1}}$. Since the $N+1$ exponents $n_0,\ldots,n_{N-1},M$ are distinct, $z^\top\widetilde v_M\ne0$, and therefore
\[
  z^\top\mathcal{H}_{N+1}(f;x)z = c'x^M(z^\top\widetilde v_M)^2<0.
\]
The necessary direction of Theorem~\ref{thm:main} now gives the claimed failure on $\mathbb{P}_{N+1}((0,\rho))$.
\end{proof}

As a special case, we also recover a real version of Belton--Guillot--Khare--Putinar's characterization of polynomial preservers \cite[Theorem 1.1]{BeltonGuillotKharePutinar}.

\begin{corollary}
\label{thm:dimension-separation}
Fix $n\geq2$, an integer $M\geq n$, and $0<\rho<\infty$.  For $\gamma>0$ and $\mathbf{c} = (c_0,\dots,c_{n-1})\in(0,\infty)^n$, define $ F_{\mathbf{c},M,\gamma}(x) := \sum_{j = 0}^{n-1}c_jx^j-\gamma x^M.$ Then the following are equivalent:
\begin{enumerate}
\item
For every $A\in\mathbb{P}_n((0, \rho))$, one has $F_{\mathbf{c},M,\gamma}[A] \succeq 0$;
\item $\displaystyle \gamma^{-1} \geq \sum_{j = 0}^{n-1} \frac{1}{c_j} \binom{M}{j}^{\!2} \binom{M-j-1}{n-j-1}^{\!2} \rho^{M-j}.$
\end{enumerate}
\end{corollary}

\section{Sharp quantitative applications}\label{sec:sharp-domination}

\subsection{Optimal domination under finite regularity}

A recurrent quantitative problem in entrywise positivity is to determine the smallest constant $K$ for which
\begin{equation}\label{eq:general-domination-problem}
  g[A]\preceq K h[A] \qquad\text{for every }A\in\mathbb{P}_n(I).
\end{equation}
Belton, Guillot, Khare, and Putinar formulated this as an extremal problem involving generalized Rayleigh quotients, and solved it for several special matrix pencils; see \cite[Sections~4 and~6]{BeltonGuillotKharePutinar}. Khare and Tao later obtained the exact value over \(\mathbb P_n^1(I)\) when \(h\) is a sum of \(n\) positive real powers and \(g\) is a single higher power, and showed that the same sharp threshold holds over \(\mathbb P_n(I)\) when the exponents of \(h\) belong to $\mathcal{A}_n$; see \cite[Theorems~1.11 and~1.12]{KhareTao}. For general analytic $g$, however, their Corollary~1.7 gives only the existence of a finite constant. They further suggested that analyticity should be replaceable by finite smoothness, but did not pursue this question.

This leaves two related issues: determining the optimal constant in \eqref{eq:general-domination-problem}, and carrying this out under finite regularity.  The next theorem answers both questions whenever the reference function $h$ lies in the strict part of the cone singled out by Theorem~\ref{thm:main}. To state the result, we denote by $\lambda_{\max}(B)$ the largest eigenvalue of a real symmetric matrix $B$.

\begin{theorem}\label{thm:sharp-smooth-domination}
Fix $n\geq1$, $0<\rho\leq\infty$ and $I = (0,\rho)$. Let $g, h\in C^{2n-2}(I)$, and suppose that for all $x\in I$, $\mathcal{H}_n(h;x)\succ0$ and $h^{(r)}(x)>0$ for all $0\leq r\leq n-1$. Define
\[
\gamma_n(h,g;x) := \lambda_{\max}\!\left( \mathcal{H}_n(h;x)^{-1/2}\mathcal{H}_n(g;x) \mathcal{H}_n(h;x)^{-1/2}\right)
\]
and
\begin{equation}\label{eq:sharp-smooth-domination-constant}
  \mathcal{K}_n(h,g;I) := \sup_{x\in I}\max\biggl\{0,\ \max_{0\leq r\leq n-1} \frac{g^{(r)}(x)}{h^{(r)}(x)},\ \gamma_n(h,g;x) \biggr\}.
\end{equation}
Then, for every $K\geq0$, the following assertions are equivalent:
\begin{enumerate}
    \item\label{item:smooth-domination-matrices} $g[A]\preceq K h[A]$ for every $A\in\mathbb{P}_n(I)$,
    \item\label{item:smooth-domination-constant} $K\geq\mathcal{K}_n(h,g;I)$.
\end{enumerate}
Therefore, a finite domination constant exists if and only if $\mathcal{K}_n(h,g;I)<\infty$.
\end{theorem}

\begin{proof}
For $K\geq0$, set $q_K := Kh-g$.  Theorem~\ref{thm:main} shows that $g[A]\preceq K h[A]$ for every $A\in\mathbb{P}_n(I)$ if and only if, for every $x\in I$,
\begin{equation}\label{eq:smooth-domination-differential-tests}
  K h^{(r)}(x)-g^{(r)}(x)\geq0 \qquad (0\leq r\leq n-1)
\end{equation}
and
\begin{equation}\label{eq:smooth-domination-hankel-test}
  K\mathcal{H}_n(h;x)-\mathcal{H}_n(g;x) \succeq 0.
\end{equation}
Since $h^{(r)}(x)>0$ for all $0\leq r\leq n-1$, condition \eqref{eq:smooth-domination-differential-tests} is equivalent to
\[
  K\geq \max_{0\leq r\leq n-1} \frac{g^{(r)}(x)}{h^{(r)}(x)}.
\]
Since $\mathcal{H}_n(h;x)\succ0$, congruence by $\mathcal{H}_n(h;x)^{-1/2}$ makes \eqref{eq:smooth-domination-hankel-test} equivalent to
\[
  K\geq\gamma_n(h,g;x).
\]
Taking the supremum on $x\in I$ and recalling that $K\geq0$ proves the result.
\end{proof}

\begin{remark}\label{rem:smooth-domination-literature}
Theorem~\ref{thm:sharp-smooth-domination} gives both an existence criterion and the optimal constant.  In particular, it gives a precise answer to the finite-regularity problem raised by Khare and Tao after \cite[Corollary~1.7]{KhareTao}: under the natural nondegeneracy conditions $\mathcal{H}_n(h;x)\succ0$ and $h^{(r)}(x)>0$ for all $0\leq r\leq n-1$, $C^{2n-2}$ regularity suffices, and a finite domination constant exists exactly when the explicit quantity \eqref{eq:sharp-smooth-domination-constant} is finite. 
\end{remark}

\subsection{Absolutely monotonic functions}

In the spirit of \cite[Theorem 1.3]{BeltonGuillotKharePutinar}, we derive the optimal domination constant $K$ in Equation \eqref{eq:general-domination-problem} in the case where $g$ and $h$ are absolutely monotonic functions. 

\begin{theorem}\label{thm:sharp-analytic-domination} 
Fix $n\geq1$ and $0<\rho<\infty$, and let $h(x) = \sum_{m = 0}^\infty b_mx^m$ and $g(x) = \sum_{m = 0}^\infty a_mx^m$ be real analytic on a neighborhood of $[0,\rho]$. Suppose that $a_m,b_m\geq0$ for all $m\geq 0$, that $b_m>0$ for at least $n$ distinct values of $m$, and that $b_m = 0\implies a_m = 0.$ Set $S := \{m\in\mathbb N_0:b_m>0\}.$ For $m\in S$, define $q_m := \frac{a_m}{b_m},$ and suppose that $(q_m)_{m\in S}$ is nondecreasing. Then $\mathcal{H}_n(h;\rho)\succ0$, the constant 
\[
  K_{n,\rho}(h,g) := \lambda_{\max}\!\left( \mathcal{H}_n(h;\rho)^{-1/2} \mathcal{H}_n(g;\rho) \mathcal{H}_n(h;\rho)^{-1/2} \right)
\]
is well-defined, and the following are equivalent for $K \geq 0$:
\begin{enumerate}
    \item $g[A]\preceq Kh[A]$ for every $A\in\mathbb{P}_n((0,\rho))$;
    \item $K\geq K_{n,\rho}(h,g)$.
\end{enumerate}
\end{theorem}

\begin{proof}
For $m\geq0$, set $v_m = (1,m,\ldots,m^{n-1})^\top.$ Since $\left(x\frac{d}{dx}\right)^r x^m = m^rx^m,$ we have, for $0<x\leq\rho$,
\begin{equation}\label{eq:analytic-Hankel-expansions}
  \mathcal{H}_n(h;x) = \sum_{m = 0}^\infty b_mx^m v_mv_m^\top, \qquad \mathcal{H}_n(g;x) = \sum_{m = 0}^\infty a_mx^m v_mv_m^\top.
\end{equation}
The series converge at $x = \rho$ by the analyticity assumption. Since $b_m>0$ for at least $n$ distinct indices, the corresponding vectors $v_m$ span $\mathbb R^n$ by the Vandermonde determinant formula. Hence
\[
  \mathcal{H}_n(h;x)\succ0 \qquad (0<x\leq\rho).
\]

Moreover, for every $0\leq r\leq n-1$, there is an $m\in S$ with $m\geq r$. Indeed, $S$ contains at least $n$ distinct nonnegative integers. Therefore
\[
  h^{(r)}(x) = \sum_{m = 0}^\infty b_m(m)_r x^{m-r}>0, \qquad 0<x\leq\rho.
\]
Thus Theorem~\ref{thm:sharp-smooth-domination} applies.

We first show that the function $x\longmapsto\gamma_n(h,g;x)$ is nondecreasing on $(0,\rho]$. By the Rayleigh quotient formula and \eqref{eq:analytic-Hankel-expansions},
\begin{align}
  \gamma_n(h,g;x)
  &= 
  \sup_{0\neq u\in\mathbb R^n}
  \frac{u^\top\mathcal{H}_n(g;x)u}
       {u^\top\mathcal{H}_n(h;x)u} = 
  \sup_{\substack{0\neq p\in\mathbb R[t]\\
                   \deg p\leq n-1}}
  \frac{\sum_{m = 0}^\infty a_m p(m)^2x^m}
       {\sum_{m = 0}^\infty b_m p(m)^2x^m}.
  \label{eq:gamma-polynomial-Rayleigh}
\end{align}
Here $p(m) = u^\top v_m$. Since $a_m = 0$ if $m\notin S$, the quotient on the right may be written as
\[
  R_p(x) := \frac{\sum_{m\in S}q_m b_m p(m)^2x^m}{\sum_{m\in S}b_m p(m)^2x^m}.
\]
Its denominator is strictly positive for every nonzero polynomial $p$ of degree at most $n-1$, since $S$ contains at least $n$ distinct points. Differentiating then gives
\[
  xR_p'(x) = \frac{\sum_{m,k\in S} b_mb_k p(m)^2p(k)^2x^{m+k} (q_m-q_k)(m-k)}
  {2\bigl( \sum_{m\in S}b_mp(m)^2x^m \bigr)^2} \geq 0,
\]
because $q_m$ is nondecreasing in $m$. Thus every $R_p$ is nondecreasing, and hence so is their supremum. Consequently,
\begin{equation}\label{eq:gamma-endpoint}
  \sup_{0<x<\rho}\gamma_n(h,g;x) = \gamma_n(h,g;\rho).
\end{equation}

We next show that $\gamma_n(h,g;x)$ controls all the ordinary derivative ratios. Fix $0\leq r\leq n-1$. Since
\[
  g^{(r)}(x) = x^{-r}\sum_{m\in S}q_m b_m(m)_r x^m, \qquad h^{(r)}(x) = x^{-r}\sum_{m\in S}b_m(m)_r x^m,
\]
we obtain
\[
  \frac{g^{(r)}(x)}{h^{(r)}(x)} = \frac{\sum_{m\in S}q_m b_m(m)_r x^m}{\sum_{m\in S}b_m(m)_r x^m}.
\]
The sequences $m\longmapsto q_m$ and $m\longmapsto(m)_r$ are both nondecreasing on $S$. Hence, the weighted Chebyshev inequality \cite[Theorem~43]{HardyLittlewoodPolya}, applied with weights $w_m=b_m(m)_rx^m$, gives
\[
  \frac{\sum_{m\in S}q_m b_m(m)_r x^m}{\sum_{m\in S}b_m(m)_r x^m} \leq \frac{\sum_{m\in S}q_m b_m(m)_r^2x^m}{\sum_{m\in S}b_m(m)_r^2x^m}.
\]
Now the polynomial $p_r(t) := (t)_r$ has degree $r\leq n-1$, and hence is an admissible test polynomial in \eqref{eq:gamma-polynomial-Rayleigh}. Therefore
\begin{equation}\label{eq:derivatives-controlled-by-gamma}
  \frac{g^{(r)}(x)}{h^{(r)}(x)} \leq \frac{\sum_{m\in S}a_m(m)_r^2x^m}{\sum_{m\in S}b_m(m)_r^2x^m} \leq \gamma_n(h,g;x).
\end{equation}

Finally, since $a_m\geq0$, Equation \eqref{eq:analytic-Hankel-expansions} gives $\mathcal{H}_n(g;x)\succeq0,$ and hence $\gamma_n(h,g;x)\geq0$. Thus Theorem~\ref{thm:sharp-smooth-domination}, along with Equations \eqref{eq:gamma-endpoint} and \eqref{eq:derivatives-controlled-by-gamma} yield
\begin{align*}
  \mathcal{K}_n(h,g;(0,\rho)) &= \sup_{0<x<\rho} \max\biggl\{ 0,\ \max_{0\leq r\leq n-1} \frac{g^{(r)}(x)}{h^{(r)}(x)},\ \gamma_n(h,g;x) \biggr\}\\
  &= \sup_{0<x<\rho}\gamma_n(h,g;x) = \gamma_n(h,g;\rho) = K_{n,\rho}(h,g).
\end{align*}
The result now follows from Theorem~\ref{thm:sharp-smooth-domination}.
\end{proof}

\subsection{Continuous mixtures and a single negative power}

We now consider a special case in which the sharp constant has a particularly simple form. For $n\geq2$ and $\alpha\in\mathbb{R}$, let $\mathcal{A}_n := \mathbb{N}_0\cup[n-2,\infty)$ and $v_\alpha := (1,\alpha,\ldots,\alpha^{n-1})^\top$. Given a compactly supported finite positive Borel measure $\mu$, define
\begin{equation}\label{eq:power-mixture-definitions}
  F_\mu(x) := \int_{\mathbb{R}}x^\alpha\,\mathrm{d}\mu(\alpha), \qquad G_\mu(x) := \int_{\mathbb{R}}x^\alpha v_\alpha v_\alpha^\top \,\mathrm{d}\mu(\alpha).
\end{equation}
Compactness of the support permits differentiation under the integral. Hence
\[
F_\mu^{(r)}(x) = \int_{\mathbb{R}}(\alpha)_r x^{\alpha-r}\,\mathrm{d}\mu(\alpha), \qquad \mathcal{H}_n(F_\mu;x) = G_\mu(x),
\]

Theorems~1.11 and~1.12 of Khare and Tao \cite{KhareTao}, generalizing earlier work of Belton, Guillot, Khare, and Putinar \cite{BeltonGuillotKharePutinar}, determine how much of a single higher power can be subtracted from a positive sum of $n$ powers. We extend this to positive mixtures of powers: when all the exponents occurring in $F_\mu$ are smaller than $M$, we determine the largest constant $c\geq0$ such that
\[
  F_\mu[A]-cA^{\circ M}\succeq0 \qquad \text{for every }A\in\mathbb{P}_n(I).
\]
The answer depends only on $M$, $I$, and the matrix $G_\mu(x)$ defined in \eqref{eq:power-mixture-definitions}.

\begin{corollary}\label{cor:single-power-mixture-threshold}
Fix $n\geq2$ and $M>n-2$. Let $\mu$ be a finite positive Borel measure with compact support contained in $\mathcal{A}_n\cap[0,M).$ Suppose that $\operatorname{supp}\mu$ contains at least $n$ distinct points, and let $F_\mu$ and $G_\mu$ be defined by \eqref{eq:power-mixture-definitions}. Set
\[
c_*(\mu,M;I) = 
  \begin{cases}
      \displaystyle \frac{1}{\rho^M v_M^\top G_\mu(\rho)^{-1}v_M},
      & 0<\rho<\infty,\\[3mm]
      0,
      & \rho = \infty.
  \end{cases}
\]
Then $F_\mu(x)-cx^M$ preserves positivity on $\mathbb{P}_n(I)$ if and only if $c\leq c_*(\mu,M;I).$
\end{corollary}

\begin{proof}
We first verify that $G_\mu(x)$ is positive definite. For $z\in\mathbb{R}^n$,
\[
  z^\top G_\mu(x)z = \int_{\mathbb{R}}x^\alpha(z^\top v_\alpha)^2 \,\mathrm{d}\mu(\alpha).
\]
If $z\neq0$, then $\alpha\mapsto z^\top v_\alpha$ is a nonzero polynomial of degree at most $n-1$. Since $\operatorname{supp}\mu$ contains at least $n$ distinct points, this polynomial cannot vanish on all of $\operatorname{supp}\mu$. Consequently, $z^\top G_\mu(x)z>0$, and hence $G_\mu(x)\succ0$.

For $b\in\mathbb{R}^n$ satisfying $b^\top v_M = 1$, the Cauchy--Schwarz inequality gives
\[
 1 = (b^\top v_M)^2 \leq \bigl(b^\top G_\mu(x)b\bigr) \bigl(v_M^\top G_\mu(x)^{-1}v_M\bigr),
\]
with equality when $b$ is proportional to $G_\mu(x)^{-1}v_M$. Therefore
\[
  \frac{1}{v_M^\top G_\mu(x)^{-1}v_M} = \min_{b\in\mathbb{R}^n,\,b^\top v_M = 1} b^\top G_\mu(x)b.
\]
Writing $p(\alpha) = b^\top v_\alpha$, the constraint $b^\top v_M = 1$ becomes $p(M) = 1$. Hence,
\begin{equation}\label{eq:mixture-threshold-minimization}
  \frac{1}{x^M v_M^\top G_\mu(x)^{-1}v_M} = \min_{\substack{p\in\mathbb{R}[\alpha],\ \deg p\leq n-1\\p(M) = 1}} \int_{\mathbb{R}}p(\alpha)^2x^{\alpha-M}\, \mathrm{d}\mu(\alpha).
\end{equation}
Since $\alpha<M$ on $\operatorname{supp}\mu$, the right-hand side is nonincreasing in $x$. Therefore, if $0<\rho<\infty$, continuity gives
\[
  \inf_{x\in(0,\rho)} \frac{1}{x^M v_M^\top G_\mu(x)^{-1}v_M} = \frac{1}{\rho^M v_M^\top G_\mu(\rho)^{-1}v_M}.
\]
If $\rho = \infty$, let $\alpha_{\max} := \max\operatorname{supp}\mu<M$. Taking $p\equiv1$ in \eqref{eq:mixture-threshold-minimization}, we obtain
\[
  0 \leq \frac{1}{x^M v_M^\top G_\mu(x)^{-1}v_M} \leq \int_{\mathbb{R}}x^{\alpha-M}\,\mathrm{d}\mu(\alpha) \leq \mu(\mathbb{R})x^{\alpha_{\max}-M} \qquad (x\geq1).
\]
The last expression tends to zero as $x\to\infty$. Thus, in both cases,
\begin{equation}\label{eq:mixture-threshold-infimum}
  c_*(\mu,M;I) = \inf_{x\in I} \frac{1}{x^M v_M^\top G_\mu(x)^{-1}v_M}.
\end{equation}

Now, for every $x\in I$, we have
\begin{equation}\label{eq:mixture-minus-power-hankel}
  \mathcal{H}_n(F_\mu-cx^M;x) = G_\mu(x)-cx^M v_Mv_M^\top.
\end{equation}
Since $G_\mu(x)$ is positive definite, multiplying the matrix on the right and left by $G_\mu(x)^{-1/2}$ shows that $\mathcal{H}_n(F_\mu-cx^M;x) \succeq 0$ if and only if
\begin{equation}\label{eq:mixture-minus-power-hankel-threshold}
  cx^M v_M^\top G_\mu(x)^{-1}v_M\leq1.
\end{equation}
It follows from Equation \eqref{eq:mixture-threshold-infimum} that the Euler--Hankel condition holds for every $x\in I$ exactly when $c\leq c_*(\mu,M;I).$

We show that this bound also implies the required ordinary derivative inequalities. Fix $0\leq r\leq n-1$ and define $p_r(\alpha) := \frac{1}{(M)_r}(\alpha)_r.$ Then $p_r$ has degree at most $n-1$ and $p_r(M) = 1$. Since $\operatorname{supp}\mu\subset\mathcal{A}_n\cap[0,M)$ and $M>n-2$, we have $0\leq p_r(\alpha)\leq1$ on $\operatorname{supp}\mu$. Using \eqref{eq:mixture-threshold-minimization}, we obtain
\begin{align*}
  \frac{1}{v_M^\top G_\mu(x)^{-1}v_M} \leq \int_{\mathbb{R}}p_r(\alpha)^2x^\alpha\,\mathrm{d}\mu(\alpha) \leq \int_{\mathbb{R}}p_r(\alpha)x^\alpha\,\mathrm{d}\mu(\alpha) = \frac{x^rF_\mu^{(r)}(x)}{(M)_r}.
\end{align*}
If $c\leq c_*(\mu,M;I)$, then
\[
  cx^M \leq \frac{1}{v_M^\top G_\mu(x)^{-1}v_M}
\]
for every $x\in I$. Hence
\[
  F_\mu^{(r)}(x)-c(M)_r x^{M-r} \geq 0, \qquad 0\leq r\leq n-1.
\]
Together with \eqref{eq:mixture-minus-power-hankel} and \eqref{eq:mixture-minus-power-hankel-threshold}, Theorem~\ref{thm:main} proves sufficiency.

Conversely, if $F_\mu-cx^M$ preserves positivity, the necessary direction of Theorem~\ref{thm:main} gives \eqref{eq:mixture-minus-power-hankel-threshold} for every $x\in I$. It follows that $c\leq c_*(\mu,M;I)$.
\end{proof}

\begin{remark}[Recovery of the finite-sum case]\label{rem:atomic-christoffel-threshold}
Suppose that $0<\rho<\infty$ and $\mu = \sum_{j = 0}^{n-1}c_{n_j}\delta_{n_j},$ where $n_j\in\mathcal{A}_n$, $0\leq n_0<\cdots<n_{n-1}<M$, and $c_{n_j}>0$. A direct computation yields
\[
  v_M^\top G_\mu(x)^{-1}v_M = \sum_{j = 0}^{n-1}\frac{1}{c_{n_j}x^{n_j}} \Bigg( \prod_{k\neq j} \frac{M-n_k}{n_j-n_k} \Bigg)^{\!\smash{2}}.
\]
Hence Corollary~\ref{cor:single-power-mixture-threshold} gives
\[
  c_*(\mu,M;(0,\rho)) =  \Bigg(\sum_{j = 0}^{n-1} \Bigg( \prod_{k\neq j} \frac{M-n_k}{n_j-n_k} \Bigg)^{\!\smash{2}}\frac{\rho^{M-n_j}}{c_{n_j}} \Bigg)^{\!\smash{-1}},
\]
which is precisely the threshold in Theorem~\ref{thm:generalized-polynomial-preservers} above. In the corresponding previously known cases, this recovers the sharp thresholds of Khare--Tao \cite[Theorems~1.11 and~1.12]{KhareTao} and Belton--Guillot--Khare--Putinar \cite[Theorem~1.1]{BeltonGuillotKharePutinar}.
\end{remark}

\subsection{Nonexistence of a finite interior test set}

Khare and Tao showed that the coefficient condition in Theorem~\ref{thm:generalized-polynomial-preservers} and Remark~\ref{rem:atomic-christoffel-threshold} can be detected by testing determinants on two sequences of rank-one matrices, and asked whether finitely many matrices could replace them; see \cite[Remark~8.5]{KhareTao}. The dependence of the sharp constant on $\rho$ gives a negative answer when the test matrices are fixed and have entries strictly inside $I$.

\begin{proposition}\label{prop:no-finite-test-set}
Fix $n\geq2$ and real powers $n_0<\cdots<n_{n-1}<M$ in the set $\mathcal{A}_n$. Suppose that $0<\rho\leq \infty$, $c_{n_0},\ldots,c_{n_{n-1}}>0$ and, for $c'\in\mathbb{R}$, define $f_{c'}(x) := \sum_{j = 0}^{n-1}c_{n_j}x^{n_j}+c'x^M.$ Then, for every finite collection $A_1,\ldots,A_L\in\mathbb{P}_n((0,\rho)),$ there exists $c'\in\mathbb{R}$ such that
\[
  f_{c'}[A_\ell]\succeq0, \qquad 1\leq\ell\leq L,
\]
but $f_{c'}[-]$ does not preserve positivity on $\mathbb{P}_n((0,\rho))$.
\end{proposition}

\begin{proof}
For $0<s\leq\rho$, set
\begin{equation}\label{eq:finite-test-critical-constant}
  \mathcal{C}(s) :=  \sum_{j = 0}^{n-1} \Bigg( \prod_{k\neq j} \frac{M-n_k}{n_j-n_k} \Bigg)^{\!2}\frac{s^{M-n_j}}{c_{n_j}}.
\end{equation}
Fix $A_1,\ldots,A_L\in\mathbb{P}_n((0,\rho))$. Since the collection is finite and all its entries are strictly smaller than $\rho$, there exists $\rho'$ such that $\rho'<\rho$ and $A_\ell\in\mathbb{P}_n((0,\rho'))$ for all $1\leq\ell\leq L$. Every exponent $M-n_j$ in \eqref{eq:finite-test-critical-constant} is strictly positive. Therefore, $\mathcal{C}(\rho')<\mathcal{C}(\rho)$, and hence $ -\mathcal{C}(\rho')^{-1} < -\mathcal{C}(\rho)^{-1}.$ Choose $c'$ satisfying
\begin{equation}\label{eq:coefficient-between-thresholds}
  -\mathcal{C}(\rho')^{-1} \leq c' < -\mathcal{C}(\rho)^{-1}.
\end{equation}

Applying the coefficient condition in Theorem~\ref{thm:generalized-polynomial-preservers} with endpoint $\rho'$, the first inequality in \eqref{eq:coefficient-between-thresholds} shows that $f_{c'}[-]$ preserves positivity on $\mathbb{P}_n((0,\rho'))$. In particular, $f_{c'}[A_\ell]\succeq0$ for all $1\leq\ell\leq L$. 
On the other hand, the second inequality in \eqref{eq:coefficient-between-thresholds}, together with the same coefficient condition applied with endpoint $\rho$, shows that $f_{c'}[-]$ does not preserve positivity on $\mathbb{P}_n((0,\rho))$.
\end{proof}

\smallskip
\noindent{\bf AI disclosure statement.}
ChatGPT 5.6 Sol by OpenAI was used to explore proof strategies for this paper and help with its writing. In particular, it identified Karlin's criterion (Lemma~\ref{lem:karlin}, Corollary~\ref{cor:karlin-psd}), which was previously unknown to the authors. All mathematical arguments were independently verified by the authors, who take full responsibility for the content.


\begin{thebibliography}{99}

\bibitem{BeltonGuillotKharePutinar}
A. Belton, D. Guillot, A. Khare, and M. Putinar,
\emph{Matrix positivity preservers in fixed dimension. I},
Adv. Math. \textbf{298} (2016), 325--368.
\href{https://doi.org/10.1016/j.aim.2016.04.016}{doi:10.1016/j.aim.2016.04.016}.

\bibitem{BGKP-SurveyI}
A. Belton, D. Guillot, A. Khare, and M. Putinar,
\emph{A panorama of positivity. I: Dimen\-sion free},
in \emph{Analysis of Operators on Function Spaces: The Serguei Shimorin Memorial Volume},
A. Aleman, H. Hedenmalm, D. Khavinson, and M. Putinar, eds.,
Trends in Mathematics, Birkh\"auser, Cham, 2019, 117--165.
\href{https://doi.org/10.1007/978-3-030-14640-5_5}
{doi:10.1007/978-3-030-14640-5\_5}.

\bibitem{BGKP-SurveyII}
A. Belton, D. Guillot, A. Khare, and M. Putinar,
\emph{A panorama of positivity. II: Fixed dimension},
in \emph{Complex Analysis and Spectral Theory},
H. G. Dales, D. Khavinson, and J. Mashreghi, eds.,
Contemp. Math. \textbf{743},
Amer. Math. Soc., Providence, RI, 2020, 109--150.
\href{https://doi.org/10.1090/conm/743/14958}
{doi:10.1090/conm/743/14958}.

\bibitem{BGKP-hankel}
A. Belton, D. Guillot, A. Khare, and M. Putinar,
\emph{Moment-sequence transforms},
J. Eur. Math. Soc. \textbf{24} (2022), no.~9, 3109--3160.
\href{https://doi.org/10.4171/JEMS/1145}{doi:10.4171/JEMS/1145}.

\bibitem{BGKP-tn}
A. Belton, D. Guillot, A. Khare, and M. Putinar,
\emph{Totally positive kernels, P\'olya frequency functions, and their transforms},
J. d'Analyse Math. \textbf{150} (2023), no.~1, 83--158.
\href{https://doi.org/10.1007/s11854-022-0259-7}{doi:10.1007/s11854-022-0259-7}.

\bibitem{belton2023negativity}
A. Belton, D. Guillot, A. Khare, and M. Putinar,
\emph{Negativity-preserving transforms of tuples of symmetric matrices},
Proc. Lond. Math. Soc. \textbf{132} (2026), no.~4, e70147.
\href{https://doi.org/10.1112/plms.70147}{doi:10.1112/plms.70147}.

\bibitem{belton2026distance}
A.~Belton, D.~Guillot, A.~Khare, and M.~Putinar,
\emph{Distance preservers for Lobachevsky space},
arXiv preprint arXiv:2608.22568 (2026).
\href{https://arxiv.org/abs/2608.22568}{arXiv:2608.22568}.

\bibitem{Bernstein}
S. Bernstein,
\emph{Sur les fonctions absolument monotones},
Acta Math. \textbf{52} (1929), no.~1, 1--66.
\href{https://doi.org/10.1007/BF02592679}
{doi:10.1007/BF02592679}.

\bibitem{BharaliHoltz}
G. Bharali and O. Holtz,
\emph{Functions preserving nonnegativity of matrices},
SIAM J. Matrix Anal. Appl. \textbf{30} (2008), no.~1, 84--101.
\href{https://doi.org/10.1137/050645075}{doi:10.1137/050645075}.

\bibitem{BodirskyKummerThom}
M. Bodirsky, M. Kummer, and A. Thom,
\emph{Spectrahedral shadows and completely positive maps on real closed fields},
J. Eur. Math. Soc. \textbf{28} (2026), no.~5, 2233--2259.
\href{https://doi.org/10.4171/JEMS/1509}{doi:10.4171/JEMS/1509}.

\bibitem{choudhury2025entrywise}
P.~Nath Choudhury and S.~Yadav,
\emph{Entrywise preservers of sign regularity},
arXiv preprint arXiv:2509.17902 (2025).
\href{https://arxiv.org/abs/2509.17902}{arXiv:2509.17902}.

\bibitem{damase2024multivariate}
S. S. Damase and A. Khare,
\emph{Multivariate transforms of total positivity},
arXiv preprint arXiv:2411.03391 (2024).
\href{https://arxiv.org/abs/2411.03391}{arXiv:2411.03391}.

\bibitem{FitzGeraldHorn}
C. H. FitzGerald and R. A. Horn,
\emph{On fractional Hadamard powers of positive definite matrices},
J. Math. Anal. Appl. \textbf{61} (1977), no.~3, 633--642.
\href{https://doi.org/10.1016/0022-247X(77)90167-6}
{doi:10.1016/0022-247X(77)90167-6}.

\bibitem{fitzgerald1995functions}
C. H. FitzGerald, C. A. Micchelli, and A. Pinkus,
\emph{Functions that preserve families of positive semidefinite matrices},
Linear Algebra Appl. \textbf{221} (1995), 83--102.
\href{https://doi.org/10.1016/0024-3795(93)00232-O}
{doi:10.1016/0024-3795(93)00232-O}.

\bibitem{Folland}
G. B. Folland,
\emph{Real Analysis: Modern Techniques and Their Applications},
2nd ed.,
John Wiley \& Sons, New York, 1999.

\bibitem{FriedlanderJoshi}
F. G. Friedlander and M. Joshi,
\emph{Introduction to the Theory of Distributions},
2nd ed.,
Cambridge University Press, Cambridge, 1998.

\bibitem{guillot2025positivity}
D. Guillot, H. Gupta, P. K. Vishwakarma, and C. H. Yip,
\emph{Positivity preservers over finite fields},
J. Algebra \textbf{684} (2025), 479--523.
\href{https://doi.org/10.1016/j.jalgebra.2025.07.016}
{doi:10.1016/j.jalgebra.2025.07.016}.

\bibitem{guillot2026entrywise}
D. Guillot, H. Gupta, P. K. Vishwakarma, and C. H. Yip,
\emph{Entrywise transforms preserving matrix positivity and nonpositivity},
J. Lond. Math. Soc. \textbf{113} (2026), no.~5, e70560.
\href{https://doi.org/10.1112/jlms.70560}{doi:10.1112/jlms.70560}.

\bibitem{guillot2012retaining}
D. Guillot and B. Rajaratnam,
\emph{Retaining positive definiteness in thresholded matrices},
Linear Algebra Appl. \textbf{436} (2012), no.~11, 4143--4160.
\href{https://doi.org/10.1016/j.laa.2012.01.013}
{doi:10.1016/j.laa.2012.01.013}.

\bibitem{guillot2015functions}
D. Guillot and B. Rajaratnam,
\emph{Functions preserving positive definiteness for sparse matrices},
Trans. Amer. Math. Soc. \textbf{367} (2015), no.~1, 627--649.
\href{https://doi.org/10.1090/S0002-9947-2014-06183-7}
{doi:10.1090/S0002-9947-2014-06183-7}.

\bibitem{guillot2015complete}
D. Guillot, A. Khare, and B. Rajaratnam,
\emph{Complete characterization of Hadamard powers preserving Loewner positivity, monotonicity, and convexity},
J. Math. Anal. Appl. \textbf{425} (2015), no.~1, 489--507.
\href{https://doi.org/10.1016/j.jmaa.2014.12.048}
{doi:10.1016/j.jmaa.2014.12.048}.

\bibitem{GKR-critG}
D. Guillot, A. Khare, and B. Rajaratnam,
\emph{Critical exponents of graphs},
J. Combin. Theory Ser. A \textbf{139} (2016), 30--58.
\href{https://doi.org/10.1016/j.jcta.2015.11.003}
{doi:10.1016/j.jcta.2015.11.003}.
See also \emph{Corrigendum: Critical exponents of graphs},
J. Combin. Theory Ser. A \textbf{223} (2026), 106222.
\href{https://doi.org/10.1016/j.jcta.2026.106222}
{doi:10.1016/j.jcta.2026.106222}.

\bibitem{GKR-sparse}
D. Guillot, A. Khare, and B. Rajaratnam,
\emph{Preserving positivity for matrices with sparsity constraints},
Trans. Amer. Math. Soc. \textbf{368} (2016), no.~12, 8929--8953.
\href{https://doi.org/10.1090/tran/6669}{doi:10.1090/tran/6669}.

\bibitem{GuillotKhareRajaratnam}
D. Guillot, A. Khare, and B. Rajaratnam,
\emph{Preserving positivity for rank-constrained matrices},
Trans. Amer. Math. Soc. \textbf{369} (2017), no.~9, 6105--6145.
\href{https://doi.org/10.1090/tran/6826}{doi:10.1090/tran/6826}.

\bibitem{HardyLittlewoodPolya}
G. H. Hardy, J. E. Littlewood, and G. P\'olya,
\emph{Inequalities},
2nd ed.,
Cambridge University Press, Cambridge, 1952.

\bibitem{Heinavaara}
O. Hein\"avaara,
\emph{Local characterizations for the matrix monotonicity and convexity
of fixed order},
Proc. Amer. Math. Soc. \textbf{146} (2018), no.~9, 3791--3799.
\href{https://doi.org/10.1090/proc/13674}
{doi:10.1090/proc/13674}.

\bibitem{Herz}
C. S. Herz,
\emph{Fonctions op\'erant sur les fonctions d\'efinies-positives},
Ann. Inst. Fourier (Grenoble) \textbf{13} (1963), no.~1, 161--180.
\href{https://doi.org/10.5802/aif.137}{doi:10.5802/aif.137}.

\bibitem{hiai2009monotonicity}
F. Hiai,
\emph{Monotonicity for entrywise functions of matrices},
Linear Algebra Appl. \textbf{431} (2009), no.~8, 1125--1146.
\href{https://doi.org/10.1016/j.laa.2009.04.001}
{doi:10.1016/j.laa.2009.04.001}.

\bibitem{Hormander}
L. H{\"o}rmander,
\emph{The Analysis of Linear Partial Differential Operators I:
Distribution Theory and Fourier Analysis},
2nd ed.,
Springer-Verlag, Berlin, 1990.

\bibitem{Horn}
R. A. Horn,
\emph{The theory of infinitely divisible matrices and kernels},
Trans. Amer. Math. Soc. \textbf{136} (1969), 269--286.
\href{https://doi.org/10.1090/S0002-9947-1969-0264736-5}
{doi:10.1090/S0002-9947-1969-0264736-5}.

\bibitem{Karlin}
S. Karlin,
\emph{Total Positivity. Vol. I},
Stanford University Press, Stanford, CA, 1968.

\bibitem{KhareSmooth}
A. Khare,
\emph{Smooth entrywise positivity preservers, a Horn--Loewner master theorem,
and symmetric function identities},
Trans. Amer. Math. Soc. \textbf{375} (2022), no.~3, 2217--2236.
\href{https://doi.org/10.1090/tran/8563}{doi:10.1090/tran/8563}.

\bibitem{KhareMatrixAnalysis}
A. Khare,
\emph{Matrix Analysis and Entrywise Positivity Preservers},
London Math. Soc. Lecture Note Ser. \textbf{471},
Cambridge University Press, Cambridge, 2022.
\href{https://doi.org/10.1017/9781108867122}
{doi:10.1017/9781108867122}.

\bibitem{KhareTao}
A. Khare and T. Tao,
\emph{On the sign patterns of entrywise positivity preservers in fixed dimension},
Amer. J. Math. \textbf{143} (2021), no.~6, 1863--1929.
\href{https://doi.org/10.1353/ajm.2021.0049}{doi:10.1353/ajm.2021.0049}.

\bibitem{mashreghi2026functional}
J. Mashreghi, M. Nasri, and P. K. Vishwakarma,
\emph{Functional calculi, positivity, and convolution of matrices},
J. Math. Anal. Appl. \textbf{565} (2027), no.~1, 130934.
\href{https://doi.org/10.1016/j.jmaa.2026.130934}
{doi:10.1016/j.jmaa.2026.130934}.

\bibitem{PolyaSzego}
G. P\'olya and G. Szeg\"o,
\emph{Aufgaben und Lehrsätze aus der Analysis. Zweiter Band:
Funktionentheorie, Nullstellen, Polynome, Determinanten, Zahlentheorie},
Springer, Berlin, 1925.
\href{https://doi.org/10.1007/978-3-662-38380-3}
{doi:10.1007/978-3-662-38380-3}.

\bibitem{Rudin}
W. Rudin,
\emph{Positive definite sequences and absolutely monotonic functions},
Duke Math. J. \textbf{26} (1959), no.~4, 617--622.
\href{https://doi.org/10.1215/S0012-7094-59-02659-6}
{doi:10.1215/S0012-7094-59-02659-6}.

\bibitem{Schoenberg}
I. J. Schoenberg,
\emph{Positive definite functions on spheres},
Duke Math. J. \textbf{9} (1942), no.~1, 96--108.
\href{https://doi.org/10.1215/S0012-7094-42-00908-6}
{doi:10.1215/S0012-7094-42-00908-6}.

\bibitem{Schur}
I. Schur,
\emph{Bemerkungen zur Theorie der beschr\"ankten Bilinearformen mit unendlich
vielen Ver\"anderlichen}, J. Reine Angew. Math. \textbf{140} (1911), 1--28.
\href{https://doi.org/10.1515/crll.1911.140.1}
{doi:10.1515/crll.1911.140.1}.

\bibitem{Vasudeva}
H. L. Vasudeva,
\emph{Positive definite matrices and absolutely monotonic functions},
Indian J. Pure Appl. Math. \textbf{10} (1979), no.~7, 854--858.

\bibitem{vishwakarma2023positivity}
P. K. Vishwakarma,
\emph{Positivity preservers forbidden to operate on diagonal blocks},
Trans. Amer. Math. Soc. \textbf{376} (2023), no.~8, 5261--5279.
\href{https://doi.org/10.1090/tran/8256}{doi:10.1090/tran/8256}.

\end{thebibliography}
\end{document}